\documentclass[10pt, reqno]{amsart}

\usepackage[linguistics]{forest}
\usetikzlibrary{trees}

\usepackage{latexsym}
\usepackage{amssymb}
\usepackage{mathrsfs}
\usepackage{amsmath}
\usepackage{fancybox,color}
\usepackage{enumerate}
\usepackage{hyperref}
\usepackage{eurosym}
\usepackage[utf8]{inputenc}

\newcommand{\dy}{\, \mathrm{d}y}
\newcommand{\dx}{\, \mathrm{d}x}
\newcommand{\dt}{\, \mathrm{d}t}

\newcommand{\dz}{\, \mathrm{d}z}

\newcommand{\kom}[1]{}
\renewcommand{\kom}[1]{{\bf [#1]}}

\numberwithin{equation}{section}

 \def\1{\raisebox{2pt}{\rm{$\chi$}}}

\newtheorem{theorem}{Theorem}[section]
\newtheorem{corollary}[theorem]{Corollary}
\newtheorem{lemma}[theorem]{Lemma}
\newtheorem{proposition}[theorem]{Proposition}

\newcommand{\CR}{{\mathcal{R}}}
\newcommand{\R}{{\mathbb R}}

\newcommand{\N}{{\mathbb N}}
\newcommand{\Z}{{\mathbb Z}}

\newcommand{\CB}{{\mathcal B}}

 \def\1{\raisebox{2pt}{\rm{$\chi$}}}
 
\newcommand{\norm}[1]{\left|\left|#1\right|\right|}

\def\vint_#1{\mathchoice%
          {\mathop{\kern 0.2em\vrule width 0.6em height 0.69678ex depth -0.58065ex
                  \kern -0.8em \intop}\nolimits_{\kern -0.4em#1}}%
          {\mathop{\kern 0.1em\vrule width 0.5em height 0.69678ex depth -0.60387ex
                  \kern -0.6em \intop}\nolimits_{#1}}%
          {\mathop{\kern 0.1em\vrule width 0.5em height 0.69678ex
              depth -0.60387ex
                  \kern -0.6em \intop}\nolimits_{#1}}%
          {\mathop{\kern 0.1em\vrule width 0.5em height 0.69678ex depth -0.60387ex
                  \kern -0.6em \intop}\nolimits_{#1}}}
\def\vintslides_#1{\mathchoice%
          {\mathop{\kern 0.1em\vrule width 0.5em height 0.697ex depth -0.581ex
                  \kern -0.6em \intop}\nolimits_{\kern -0.4em#1}}%
          {\mathop{\kern 0.1em\vrule width 0.3em height 0.697ex depth -0.604ex
                  \kern -0.4em \intop}\nolimits_{#1}}%
          {\mathop{\kern 0.1em\vrule width 0.3em height 0.697ex depth -0.604ex
                  \kern -0.4em \intop}\nolimits_{#1}}%
          {\mathop{\kern 0.1em\vrule width 0.3em height 0.697ex depth -0.604ex
                  \kern -0.4em \intop}\nolimits_{#1}}}

\newcommand{\intav}{\vint}
\newcommand{\aveint}[2]{\mathchoice%
          {\mathop{\kern 0.2em\vrule width 0.6em height 0.69678ex depth -0.58065ex
                  \kern -0.8em \intop}\nolimits_{\kern -0.45em#1}^{#2}}%
          {\mathop{\kern 0.1em\vrule width 0.5em height 0.69678ex depth -0.60387ex
                  \kern -0.6em \intop}\nolimits_{#1}^{#2}}%
          {\mathop{\kern 0.1em\vrule width 0.5em height 0.69678ex depth -0.60387ex
                  \kern -0.6em \intop}\nolimits_{#1}^{#2}}%
          {\mathop{\kern 0.1em\vrule width 0.5em height 0.69678ex depth -0.60387ex
                  \kern -0.6em \intop}\nolimits_{#1}^{#2}}}

\newcommand{\re}{\mathbb{R}}

\begin{document}
\title[
Centered fractional maximal radial function
]
{Regularity of the centered fractional maximal function on radial functions
}
\author{David Beltran and Jos{\'e} Madrid}

\date{\today}

\address{David Beltran: Department of Mathematics, University of Wisconsin, 480 Lincoln Drive, Madison, WI, 53706, USA.}
\email{dbeltran@math.wisc.edu}
	
\address{José Madrid: Department of  Mathematics,  University  of  California,  Los  Angeles (UCLA),  Portola Plaza 520, Los  Angeles,
California, 90095, USA}
\email{jmadrid@math.ucla.edu}


\maketitle
\begin{abstract}
We study regularity properties of the centered fractional maximal function $M_{\beta}$. More precisely, we prove that the map $f \mapsto |\nabla M_\beta f|$ is bounded and continuous from $W^{1,1}(\R^d)$ to $L^q(\R^d)$ in the endpoint case $q=d/(d-\beta)$ if $f$ is a radial function. For $d=1$, the radiality assumption can be removed. This corresponds to the counterparts of known results for the non-centered fractional maximal function. The main new idea consists in relating the centered and non-centered fractional maximal function at the derivative level.  
\end{abstract}

\section{Introduction}
Given $f \in L^1_{loc}(\R^d)$ and $0 \leq \beta < d$, the centered fractional Hardy-Littlewood maximal operator $M_{\beta}$ is defined by
\begin{equation*}
M_{\beta} f(x):=\sup_{r \geq 0}
\frac{r^\beta}{|B(x,r)|}\int_{B(x,r)}|f(y)| \dy\, 
\end{equation*}
for every $x\in\R^d\,$. The non-centered version of $M_{\beta}$, denoted by $\widetilde{M}_\beta$, is defined by taking the supremum over all balls containing $x$. The non-fractional case $\beta=0$ corresponds to the classical maximal function and is denoted  by $M=M_0$ and $\widetilde{M}=\widetilde{M}_0$.


Questions regarding the regularity of maximal functions started with the work of Kinnunen \cite{Kinnunen1997}, who showed that if $f \in W^{1,p}(\R^d)$ with $1<p< \infty$, then $M f \in W^{1,p}(\R^d)$ and
$$
|\nabla M f(x)| \leq M (|\nabla f|)(x)
$$
almost everywhere in $\R^d$. The fractional case was first studied by Kinnunen and Saksman \cite{KS2003}, who noticed that Kinnunen's result extends to $0 < \beta < d$, with $M_\beta f \in W^{1,q}(\R^d)$ for $\frac{1}{q}=\frac{1}{p}-\frac{\beta}{d}$ and $1 < p < \infty$, and moreover, showed that if $f \in L^p(\R^d)$ with $1 < p < d$ and $1 \leq \beta < d/p$, then
\begin{equation}\label{eq:KS}
|\nabla M_\beta f (x)| \leq C M_{\beta - 1} f(x)
\end{equation}
almost everywhere in $\R^d$. The continuity of the map $f\mapsto M_{\beta}f$ from $W^{1,p}(\R^d)$ to $W^{1,q}(\R^d)$ was established by Luiro in \cite{Luiro2007} for $\beta=0$ and $1 < p <\infty$, although it immediately generalises to the case $0<\beta<d$; of course the sublinearity of $M_\beta$ ceases to hold at the derivative level, so the continuity of this map on Sobolev spaces does not immediately follow from its boundedness. It is noted that these results continue to work for the non-centered counterparts $\widetilde{M}$ and $\widetilde{M}_\beta$. This is strongly attached to the case $p>1$, as the mentioned results rely on the Lebesgue space boundedness of the aforementioned maximal functions.

In the case $p=1$, one cannot expect $M_\beta f \in W^{1,\frac{d}{d-\beta}}$ if $f \in W^{1,1}(\R^d)$ for any $0 \leq \beta < d$, as $M_\beta$ fails to be bounded at the level of Lebesgue spaces. However, one may still ask whether $M_\beta f$ is weakly differentiable and whether the map $f \mapsto |\nabla M_\beta f|$ is bounded and continuous from $W^{1,1}(\R^d)$ to $L^{\frac{d}{d-\beta}}(\R^d)$. This question was originally posed for $\beta=0$, is commonly referred to as the $W^{1,1}(\R^d)$-problem, and it is well known that its study significantly differs from the centered to the non-centered case. In this paper, we concern ourselves with $\beta > 0$ and the centered fractional maximal function; for $\beta =0$, see \cite{Tanaka2002, AP2007} and \cite{CMP2017} for boundedness and continuity results for $\widetilde{M}$ if $d=1$, \cite{Kurka2010} for boundedness for $M$  if $d=1$, and \cite{Luiro2017} for boundedness for $\widetilde{M}$ over radial functions if $d>1$.

In the strictly fractional case, it was noted by Carneiro and the second author \cite{CM2015} that the interesting endpoint case $p=1$ corresponds to the range $0 < \beta < 1$. Indeed, if $1 \leq \beta < d$, \eqref{eq:KS} and the Gagliardo--Nirenberg--Sobolev inequality immediately yield, together with the bounds of $M_{\beta-1}$, that
$$
\| \nabla M_\beta f \|_q \leq C \| M_{\beta- 1} f \|_q \leq C \| f \|_{\frac{d}{d-1}} \leq C \| \nabla f \|_1
$$
holds for $q=\frac{d}{d-\beta}$. This settles the boundedness of the map $f \mapsto |\nabla M_\beta f|$ from $W^{1,1}(\R^d)$ to $L^q(\R^d)$ if $1 < \beta < d$. The continuity of this map between such function spaces was recently established by the authors in \cite{BM2019}. In this situation, the same analysis continues to work for $\widetilde{M}_\beta$.

If $0 < \beta < 1$, the corresponding boundedness result for $\widetilde{M}_\beta$ was established by Carneiro and the second author \cite{CM2015} if $d=1$, whilst its continuity at the derivative level was later shown by the second author \cite{Madrid2017}. Our first result is a counterpart for the centered case.

\begin{theorem}\label{thm:d=1}
Let $0 < \beta < 1$ and $q=1/(1-\beta)$. If $f \in W^{1,1}(\R)$, then $M_\beta f$ is differentiable almost everywhere and
\begin{equation}\label{eq:endpoint sobolev d=1}
    \| (M_\beta f)' \|_{L^{q} (\R)} \leq C \| f' \|_{L^1(\R)},
\end{equation}
where $C= 2^{-3\beta -2+ \frac{4}{\beta}} 3^{\frac{2(1-\beta)^2}{\beta}}$. Moreover, the map $f \mapsto | (M_\beta f)'|$ is continuous from $W^{1,1}(\R)$ to $L^q(\R)$.
\end{theorem}

The continuity statement in the above theorem was already established by the authors in \cite[Theorem 1.3]{BM2019} under the boundedness hypothesis. The new contribution is therefore the almost everywhere differentiability of $M_\beta f$ and the bound \eqref{eq:endpoint sobolev d=1}. It turns out that special features of $\beta>0$ and $d=1$ allow to analyse $M_\beta$ in a similar way to $\widetilde{M}_\beta$, using some of the arguments introduced by Luiro and the second author \cite{LM2017} in the higher-dimensional radial problem.

Our main result is in higher dimensions. More precisely, it was shown by Luiro and the second author \cite{LM2017} that if $0 < \beta < 1$, the bound $\| \nabla \widetilde M_\beta f \|_{\frac{d}{d-\beta}} \leq C \| \nabla f \|_1$ holds if $f \in W^{1,1}(\R^d)$ is a radial function. The continuity of this map was established by the authors in \cite{BM2019}. Here, we prove the corresponding results for $M_\beta$.

\begin{theorem}\label{Main Theorem}
Let $0<\beta<1$ and $q=d/(d-\beta)$. If $f \in W^{1,1}_{\mathrm{rad}}(\R^d)$, then $M_\beta f$ is differentiable almost everywhere and there exists a constant $C=C(d,\beta)>0$ such that
\begin{equation}\label{eq:sobolev bound}
\|\nabla M_{\beta}f\|_{L^{q}(\R^d)}\leq C\|\nabla f\|_{L^{1}(\R^d)}.
\end{equation}
Moreover, the map $f \mapsto |\nabla M_\beta f|$ is continuous from $W^{1,1}_{\mathrm{rad}}(\R^d)$ to $L^q(\R^d)$.
\end{theorem}

In contrast to Theorem \ref{thm:d=1}, we require a more refined analysis of the operator $M_\beta$ than that performed for $\widetilde{M}_\beta$ in \cite{LM2017}. A key new fundamental idea consists in establishing a relation between $M_\beta$ and $\widetilde{M}_\beta$ at the derivative level. Whilst $M_\beta f(x) \leq \widetilde{M}_\beta f(x) \leq 2^{d-\beta} M_\beta f(x)$, this comparability ceases to hold for the derivatives. However, we are able to establish an inequality allowing to control $\nabla M_\beta f(x)$ by $\nabla \widetilde{M}_\beta f(x)$ and some additional terms (see Lemma \ref{lemma radios grandes}). This inequality is one of the main ingredients in the proof of Theorem \ref{Main Theorem}, and allows to exploit non-centered techniques. We believe this is the first time that such a connection between centered and non-centered maximal functions has been made in the study of regularity questions. Whilst our relation is only valid for $\beta >0$, we hope that this new perspective could also be useful to better understand regularity properties for $M$.

It is noted that the main result in Theorem \ref{Main Theorem} is the endpoint Sobolev bound \eqref{eq:sobolev bound}. The continuity statement can be deduced via a similar scheme to the one used by the authors \cite{BM2019} in the non-centered case, together with a new idea recently introduced by González-Riquelme \cite{GR}. Interestingly, he obtained endpoint Sobolev results for a version of $\widetilde{M}_\beta f$ defined on the sphere $S^{d-1}$ if $f$ is polar; see also \cite{CGR} for similar results when $\beta=0$. 

Further interesting results concerning regularity of fractional maximal functions have been obtained recently; we refer the interested reader to \cite{RSW} for the boundedness of fractional maximal functions on domains (see also \cite{HKKT2015}), to \cite{BRS2018} for results on lacunary and smooth variants of $M_\beta$, as well as fractional spherical maximal functions, and to \cite{Saari2016} for Poincaré inequalities for $\widetilde{M}_\beta$.
\\

\textit{Structure of the paper.} Section \ref{sec:derivative} is devoted to the representation of the derivative of the maximal function and its almost everywhere differentiability. In Section \ref{sec:d=1}, the proof of Theorem \ref{thm:d=1} is presented. Section \ref{sec:relation} contains preliminary results that will feature in the proof of the main theorem, and in particular the key relation between $M_\beta$ and $\widetilde{M}_\beta$ at the derivative level. The boundedness part of Theorem \ref{Main Theorem} is presented in Section \ref{sec:main theorem}, whilst the continuity part is presented in Section \ref{sec:continuity}.
\\

\textit{Acknowledgements.} The second author would like to thank the Analysis group of the UW-Madison for the hospitality during his visit in September 2019, where part of this research was conducted.


\section{The derivative of the fractional maximal function}\label{sec:derivative}

We start introducing some notation. The value of the Lebesgue exponent $q$ will always be $q=d/(d-\beta)$. Given a measurable set $E \subset \R^d$, $\chi_E$ denotes the characteristic function of $E$ and $E^c:=\R^d \backslash E$  its complementary set in $\R^d$. For $c \in \R$ , we denote by $cE$ the concentric set to $E$ dilated by $c$. The integral average of $f \in L^1_{loc}(\R^d)$ over $E$ is denoted by $f_E=\intav_E f$. The notation $A \lesssim B$ is used if there exists $C>0$ such that $A \leq C B$, and similarly $A \gtrsim B$ and $A \sim B$. The volume of the $d$-dimensional unit ball is denoted by $\omega_d$ and the $(d-1)-$ dimensional Hausdorff measure of the sphere $ S^{d-1}$ is denoted by $\sigma_d$. The weak derivative of $f$ is denoted by $\nabla f$.

 Fix $0\leq \beta<d$. Given a function $f \in W^{1,1}(\R^d)$ and a point $x \in \R^d$, define the families of \textit{good balls} for $f$ at $x$ as
 \begin{equation}\label{def:good balls}
      \mathcal{B}^{\beta}_{x}(f)
=\,\mathcal{B}_x^\beta:=\bigg{\{}B(x,r):
M_{\beta}f(x)=r^{\beta}\intav_{B(x,r)}|f(y)|\dy\,\bigg{\}}
 \end{equation}
 and
  \begin{equation*}
      \widetilde{\mathcal{B}}^{\beta}_{x}(f)
=\,\widetilde{\mathcal{B}}_x^\beta:=\bigg{\{}B(z,r) : x\in\bar{B}(z,r)\,,\, \, 
\widetilde{M}_{\beta}f(x)=r^{\beta}\intav_{B(x,r)}|f(y)|\dy\,\bigg{\}}.
 \end{equation*}

Note that $\CB_x^\beta (f) \neq \emptyset$ and $\widetilde{\CB}_x^\beta (f) \neq \emptyset$ for all $x \in \R^d$ if $f \in L^1(\R^d)$.  Moreover, $\CB_x^\beta(f)$ and $\widetilde{\CB}_x^\beta(f)$ are compact sets. The families of \textit{good radii} are denoted by $\mathcal{R}_x^\beta$ and $\widetilde{\mathcal{R}}_x^\beta$ respectively. Note that, as a consequence of the Lebesgue differentiation theorem, $0\notin \mathcal{R}_x^{\beta}$ and $0\notin \widetilde{\mathcal{R}}_x^{\beta}$.


We start recalling some useful facts observed for $\widetilde{M}_\beta$ in \cite{LM2017}, and establish the analogues for $M_\beta$. First, note the following result, which follows by a simple change of variables.

\begin{proposition}[Proposition 2.1 \cite{LM2017}]\label{afin map 1}
Let $f\in W^{1,1}(\R^d)$ and let $\{L_i\}_{i \in \N}$ be a family of affine maps $L_{i}(y)=a_{i}y+b_{i}$, $a_{i}\in\mathbb{R}$, $b_{i}\in\R^d$. Let $\{h_{i}\}_{i\in\N}\subset\mathbb{R}$ be a sequence such that $h_{i}\to0$ as $i\to\infty$, 
$$
\lim_{i\to\infty}\frac{L_{i}(y)-y}{h_{i}}=g(y)
\,\,\,\,\text{ and }\,\,\,\,
\lim_{i\to\infty}\frac{a_{i}^{\beta}-1}{h_{i}}=\gamma,
$$
where $\gamma\in\re$, $g:\R^d\to\R^d\,$. If $B$ is a ball in $\R^d$ and $B_i:=L_i(B)$, it holds that
 \begin{align*}
&\lim_{i\to\infty}\frac{1}{h_{i}}\left(r^{\beta}_{i}\intav_{B_{i}}|f(y)| \dy-r^{\beta}\intav_{B}|f(y)| \dy\right)\\
&\ \ \ =\,r^{\beta}\intav_{B}\nabla |f|(y)\cdot g(y) \dy+\gamma \, r^{\beta}\intav_{B}|f(y)|\dy\,,
\end{align*}
where $r$ denotes the radius of $B$ and $r_{i}$ the radius of $B_{i}$ for every $i \in \N$.
\end{proposition}

The above proposition can be used to obtain the following lemma (the non-centered statement corresponds to \cite[Lemma 2.2]{LM2017}).

\begin{lemma}\label{afin map 2}
Let $f\in W^{1,1}(\R^d)$, $x\in\R^d$, $\delta>0$, and let $L_{h}(y)=a_{h}y+b_{h}$, $h\in[-\delta,\delta]$, be affine mappings such that
$$
\lim_{h\to0}\frac{L_{h}(y)-y}{h}=g(y)
\,\,\text{ and }\,\,
\lim_{h\to 0}\frac{a_{h}^{\beta}-1}{h}=\gamma.
$$
Assume that $B\in \mathcal{B}_{x}^\beta$ and $x= L_{h}(x)$. Then
\begin{equation}\label{eq99}
0=r^{\beta}\intav_{B}\nabla |f|(y)\cdot g(y) \dy+\gamma M_{\beta}f(x),
\end{equation}
where $r$ denotes the radius of $B$. Moreover, if $\widetilde{B} \in \widetilde{\mathcal{B}}_x^\beta$ and $x \in L_h(\overline{ \widetilde{B}})$, \eqref{eq99} holds for $\widetilde{B}$ and $\widetilde{M}_\beta$.
\end{lemma}
\begin{proof}
Let $r_h$ denote the radius of $L_h(B)$. By Proposition \ref{afin map 1} the right hand side of (\ref{eq99}) equals to 
\begin{equation*}
\lim_{h\to 0}\frac{1}{h}\left( r_h^{\beta}\intav_{L_h(B)}|f(y)| \dy-r^{\beta}\intav_{B}|f(y)| \dy\right)\,=:\lim_{h\to 0}\frac{1}{h}s_h\,.
\end{equation*}
Since $B\in\mathcal{B}_x^\beta$ and $x$ is also the center of the ball $L_h(B)$ for all $h$, it follows that $s_h\leq 0$ for all $h$. Since $h$ can take positive and negative values, the existing limit must be equal to zero. The corresponding result in the non-centered situation follows similarly.
\end{proof}

This implies as a consequence the following corollary, which will play a crucial rôle in relating $M_\beta$ and $\widetilde{M}_\beta$ at the derivative level.

\begin{corollary}\label{coro good eq}
Let $f\in W^{1,1}(\R^d)$, $x\in\R^d$ and $B=B(x,r) \in \CB_{x}^\beta$. Then
\begin{equation*}\label{nice identity}
r^{\beta}\intav_{B}\nabla |f|(y)\cdot x \dy=r^{\beta}\intav_{B}\nabla |f|(y)\cdot y\dy+\beta M_{\beta}f(x)
\end{equation*}
and the same holds for $\widetilde{B} \in \widetilde{\CB}_x^\beta$ and $\widetilde{M}_\beta$.
\end{corollary}

\begin{proof}
Given $\delta>0$, define $L_{h}(y)=y+h(y-x)$ for $h \in [-\delta,\delta]$. The desired identity follows from Lemma \ref{afin map 2} after noting that $g(y)=y-x$, $\gamma=\beta$ and that if $B \in \mathcal{B}_x^\beta$, then $x=L_h(x)$ is the center of $L_h(B)$ and if $\widetilde{B} \in \widetilde{\CB}_x^\beta$, then $x=L_h(x) \in L_h(\overline{\widetilde{B}})$.
\end{proof}

Our next goal is to show that $M_\beta$ is differentiable almost everywhere. To this end, it is useful to introduce the following object. Given $\varepsilon>0$, define the truncated fractional maximal function as
$$
M_\beta^\varepsilon f (x):= \sup_{r \geq \varepsilon} r^\beta \intav_{B(x,r)} |f|.
$$

The following observation follows from  a straightforward adaption to the fractional setting of the arguments in Hajlasz--Maly \cite[Lemma 8]{HM2010}.

\begin{lemma}\label{lemma:M truncated}
Let $0 \leq \beta < d$ and $\varepsilon>0$. If $f \in L^1(\R^d)$,
$$
|M_\beta^\varepsilon(x)-M_\beta^\varepsilon(y)| \leq \frac{d-\beta}{\varepsilon} |x-y| \big( M_\beta^\varepsilon f(x) + M_\beta^\varepsilon f(y) \big) \leq \frac{2(d-\beta)}{\omega_d \varepsilon^{d+1-\beta}} \| f \|_{1} |x-y|.
$$
Consequently, $M_\beta^\varepsilon$ is Lipschitz continuous, and in particular differentiable almost everywhere.
\end{lemma}


The following lemma is essentially contained in Kinnunen \cite{Kinnunen1997}; details are provided for the reader's convenience.

\begin{lemma}\label{lemma:M continuous}
Let $0 < \beta < d$. If $f\in L^{1}(\R^d)$ is a continuous function, then $M_{\beta}f$ is also continuous.
\end{lemma}
\begin{proof}
Given $h \in \R^d$, let $f_h(x):=f(x+h)$. Fix $x \in \R^d$. Given $\epsilon>0$, there exists $r_{\epsilon}>0$ such that
\begin{eqnarray*}
r^{\beta}\intav_{B(x,r)}|f-f_h|\leq \frac{\|f-f_h\|_1}{\omega_d \,  r^{d-\beta}}\leq \frac{2\|f\|_1}{\omega_d \, r^{d-\beta}}<\epsilon
\end{eqnarray*}
for every $r>r_\epsilon$. On the other hand, if $0<r\leq r_{\epsilon}$, there exists $\delta>0$ such that
\begin{align*}
r^{\beta}\intav_{B}|f-f_h| & \,\leq \, \frac{1}{\omega_d} \|f_h-f\|_{L^{q'}(B)} \\ 
&\leq \, \frac{1}{\omega_d} \|f_h-f\|^{1/q'}_{1}\sup_{y\in B(x,r_\epsilon)}|f(y)-f_h(y)|^{(q'-1)/q'}\\
&\leq \, \frac{1}{\omega_d} 2^{1/q'}\|f\|^{1/q'}_1 \epsilon^{(q'-1)/q'} 
\end{align*}
for every $|h|<\delta$, where $q=d/(d-\beta)$. Then, by sublinearity,
$|M_{\beta}f(x)-(M_{\beta}f)_h(x)|\leq M_{\beta}(f-f_h)(x)\leq \widetilde \epsilon$ for every $|h|<\delta$. Therefore $M_{\beta}f$ is continuous at $x$.
\end{proof}

These two lemmas can be combined to show the a.e. differentiability if $d=1$.

\begin{lemma}\label{lemma:derivative d=1}
Let $d=1$ and $0 < \beta < 1$. For any $f\in W^{1,1}(\R)$, $M_{\beta}f$ is differentiable almost everywhere. Moreover, if $M_\beta f$ is differentiable at $x$ and $B=B(x,r) \in \CB_x^\beta$, one has
\begin{equation}\label{eq:derivative d=1}
(M_\beta f)'(x)= r^\beta \intav_B |f|'(y)  \dy.
\end{equation}
\end{lemma}
\begin{proof}
As $f \in W^{1,1}(\R)$, the function $f$ is continuous, and by Lemma \ref{lemma:M continuous} $M_\beta f$ is also continuous. Let $I_{n}=[n,n+1]$ and write $\R = \cup_{n \in \Z} I_n$. For each fixed $n \in \Z$, the continuity of $M_\beta f$ ensures that there exists $C_n>0$ such that $M_{\beta}f (y) \geq C_n > 0$ for all $y \in I_n$. In particular, one has
\begin{eqnarray*}
r^\beta_{y}\|f\|_{\infty}\geq M_{\beta}f(y)\geq C_n
\end{eqnarray*}
for all $y\in I_n$, where $r_{y}$ is a good radius for $y$. This implies that $M_{\beta}f(y)=M^{\varepsilon_n}_{\beta}f(y)$ for all $y\in I_n$, where $\varepsilon_n=\left( C_n/\|f\|_{\infty}\right)^{1/\beta}$. Then, by Lemma \ref{lemma:M truncated}
one has that $M_\beta f$ is differentiable almost everywhere in $I_n$. As the countable union of sets of measure zero is a set of measure zero, one concludes that $M_\beta f$ is differentiable almost everywhere in $\R$. 

For the second part, let $B=B(x,r)\in\mathcal{B}_x^\beta\,$ and  $B_{h}:=B(x+h,r)$. It then follows that 
\begin{eqnarray*}
\lim_{h\to 0} \frac{M_{\beta}f(x+h)-M_{\beta}f(x)}{h}
&\geq& \lim_{h\to 0}\frac{r^{\beta}\intav_{B_{h}}|f(y)| \dy-r^{\beta}\intav_{B}|f(y)|\dy}{h}\\
&=&r^{\beta}\intav_{B}|f|'(y) \dy\\
&=& \lim_{h\to 0}\frac{r^{\beta}\intav_{B}|f(y)| \dy-r^{\beta}\intav_{B_{-h}}|f(y)|\dy}{h}\\
&\geq&  \lim_{h\to 0} \frac{M_{\beta}f(x)-M_{\beta}f(x-h)}{h}\,,
\end{eqnarray*}
as desired.
\end{proof}

The $\text{a.e.}$ differentiability and the representation \eqref{eq:derivative d=1} will play an important rôle in establishing the endpoint Sobolev bound \eqref{eq:endpoint sobolev d=1}. This is in contrast with the analogous result for $\widetilde{M}_\beta$ in \cite{CM2015}, in which the bound $\| (\widetilde{M}_\beta f)' \|_q \lesssim \| f' \|_1$ was obtained without using the version of \eqref{eq:derivative d=1} for $\widetilde{M}_\beta$. In fact, a slightly stronger result concerning the $q$-variation of $\widetilde{M}_\beta f$ was proved in \cite{CM2015}, which via a classical result of Riesz allowed to deduce that $\widetilde{M}_\beta f$ is absolutely continuous if $f \in W^{1,1}(\R)$.

If $d>1$, it is possible to adapt the proof of Lemma \ref{lemma:derivative d=1} if the function $f$ is radial. To this end, define the auxiliary operator
\begin{equation*}\label{eq:MI}
    M^{I}_\beta f(x) :=  \sup_{r\leq |x|/4}
 r^\beta \intav_{B(x,r)}|f(y)|\dy.
\end{equation*}
This operator was introduced by Luiro \cite{Luiro2017} in the non-centered $\beta=0$ case, and will essentially behave as a one-dimensional maximal function when acting over radial functions; see the end of Section \ref{sec:relation} for further discussion.

\begin{lemma}\label{basic}
Let $d>1$ and $0 < \beta < 1$. For any radial function $f \in W^{1,1}(\R^d)$, $M_\beta f$ is differentiable almost everywhere. 

Moreover, let $x \in \R^d$ and assume $M_{\beta}f$ and $\widetilde M_{\beta}f$ are differentiable at $x$.
\begin{enumerate}[i)]
\item 
If $B=B(x,r) \in \CB_x^\beta$, then 
\begin{equation}\label{Luiro formula}
\nabla M_{\beta}f(x)=r^{\beta}\intav_{B}\nabla|f|(y)\dy\,,
\end{equation}
and the same holds for $\widetilde{M}_\beta$ and $\widetilde{B} \in \widetilde{\CB}_x^\beta$.
\item If $B\in\mathcal{B}_x^\beta$, then 
\begin{equation}\label{char}
\int_{B}|f(y)| \dy\,=\,-\frac{1}{\beta}\int_{B}\nabla |f|(y)\cdot (y-x)\dy\,,
\end{equation}
and the same holds for $\widetilde{B} \in \widetilde{\CB}_x^\beta$.
\item 
If $f$ is radial, $x\neq0$ and $\nabla M_{\beta}f(x)\neq0$, then
\begin{equation*}
\frac{\nabla M_{\beta}f(x)}{|\nabla M_{\beta}f(x)|}=\frac{\pm x}{|x|}\,.
\end{equation*}
\item  If $\nabla \widetilde M_{\beta}f(x)\neq0$, then $x\in\partial \widetilde B$ for any $\widetilde B= B(z,r)\in \mathcal{\widetilde B}_x^\beta$ and
$$
\frac{\nabla \widetilde{M}_\beta f (x)}{|\nabla \widetilde{M}_\beta f (x)|} = \frac{z-x}{|z-x|}.
$$
\end{enumerate}
\end{lemma}

\begin{proof}
We first prove that $M_\beta f$ is $\text{a.e.}$ differentiable. To this end, let $\R^d \backslash \{0\}=\cup_{n \in \Z} A_n$, where $A_n:=\{ x \in \R^d : 2^{n} \leq |x| \leq 2^{n+1}\}$. It suffices to show that $M_\beta f$ is a.e. differentiable on each $A_n$. For each $n \in \Z$, one may write $A_n= I_n \cup I_n^c$, where $I_n:=\{x \in A_n : M_\beta f (x) = M_\beta^I f(x)\}$. On $I_n$, by definition of $M_\beta^I$, one has $M_\beta f (x) = M_\beta^I f(x)=M_\beta^I (f\chi_{5/4 A_n})(x)$. As $f \in W^{1,1}(\R^d)$ is radial, $f$ is continuous on any annulus and therefore $M_\beta^I$ is continuous on $A_n$. One can then argue as in the proof of Lemma \ref{lemma:derivative d=1} to show that $M_\beta f$ is $\text{a.e.}$ differentiable on $I_n$. On $I_n^c$, one has $M_\beta f (x) = M_\beta^{2^{n-2}} f (x)$, which by Lemma \ref{lemma:M truncated} is $\text{a.e.}$ differentiable on $I_n^c$, concluding the argument.

Assume next that $M_\beta f (x)$ is differentiable at $x$. Item i) follows analogously to Lemma \ref{lemma:derivative d=1}, whilst item ii) is an immediate consequence of Corollary \ref{coro good eq}. Item iii) just follows from the radiality assumption on $f$. Finally, item iv) was established in \cite[Lemma 2.4]{LM2017}.
\end{proof}


\section{The case $d=1$}\label{sec:d=1}

In this section, the proof of Theorem \ref{thm:d=1} is provided. It is noted that we have already established the $\text{a.e.}$ differentiability of $M_\beta f$ in Lemma \ref{lemma:derivative d=1}. Moreover, the continuity result in the statement follows from \cite[Theorem 1.3]{BM2019} once the inequality \eqref{eq:endpoint sobolev d=1} is established.

We turn then into establishing \eqref{eq:endpoint sobolev d=1}. A key observation is the following.

\begin{proposition}\label{prop comparable radii}
Let $0 < \beta < d$ and $f\in L^{1}_{loc}(\R^d)$. Let $x_1, x_2 \in \R^d$ be such that there exist $B_1:=B(x_1,r_1) \in \CB_{x_1}^\beta$ and $B_2:=B(x_2,r_2) \in \CB_{x_2}^\beta$ with $B_1 \cap B_2 \neq \emptyset$ and $C^{-1}|f|_{B_1}\leq |f|_{B_2}\leq C|f|_{B_1}$ for some finite constant $C>0$. Then
\begin{equation*}
1 \leq \frac{\max{\{r_1,r_2\}}}{\min{\{r_1,r_2\}}} \leq C^{\frac{1}{\beta}} 3^{\frac{d-\beta}{d}}.
\end{equation*}
\end{proposition}
\begin{proof}
We assume without loss of generality that $r_1\geq r_2$. This implies that $B(x_1,r_1)\subseteq B(x_2,3r_1)$. In particular, as $B_2 \in \mathcal{B}_{x_2}^\beta$, one trivially has that
$$
M_\beta f(x_2)=r^\beta_2|f|_{B_2}\geq (3r_1)^\beta|f|_{B(x_2,3r_1)}.
$$
This and the fact that $B(x_1,r_1)\subseteq B(x_2,3r_1)$, readily imply
\begin{equation*}
|f|_{B_2}\geq 3^\beta\left(\frac{r_1}{r_2}\right)^{\beta}|f|_{B(x_2,3r_1)} \geq  3^{\beta-d}\left(\frac{r_1}{r_2}\right)^{\beta}|f|_{B(x_1,r_1)} \geq \frac{1}{C} 3^{\beta-d}\left(\frac{r_1}{r_2}\right)^{\beta}|f|_{B_2}.
\end{equation*}
Then we conclude that
\begin{equation*}
1\leq \frac{r_1}{r_2}\leq C^{\frac{1}{\beta}}3^{\frac{d-\beta}{\beta}}
\end{equation*}
where the lower bound simply follows as $r_1 \geq r_2$.
\end{proof}

The above proposition provides a comparability among good radii for intersecting good balls with comparable non-fractional average. Consequently, this provides a uniform lower bound for all $r_x \in \CR_x^\beta$ over $x \in \R^d$ such that $\cap \, B (x,r_x)\neq\emptyset$. Note that this does no longer work for $\beta=0$. Interestingly, the availability of such a lower bound when $\beta>0$ will allow us to carry a similar analysis to that performed for $\widetilde{M}_\beta^I$ in the higher dimensional, radial case in \cite{LM2017}. In particular, Lemma 2.10 in \cite{LM2017} adapts as follows, which will allow to exploit Proposition \ref{prop comparable radii}.

\begin{lemma}\label{lemma2 d=1}
Suppose that $f\in W^{1,1}_{loc}(\R)$, $0<\beta<1$, $B_x=B(x,r_x)\in \mathcal{B}^{\beta}_x$ for some $x\in\R$, and
\begin{equation*}
E_x:=\{z \in B_x\,:\,\frac{1}{2}|f|_{B_x}\leq |f(z)|\leq 2|f|_{B_x}\}\,.
\end{equation*}
Then
\begin{equation*}
|(M_\beta f)'(x)|\leq 4 \, r_x^\beta \intav_{B_x}| f'(z)|\, \chi_{E_x}(z) \dz\,
\end{equation*}
whenever $M_\beta f$ is differentiable at $x$.
\end{lemma}

\begin{proof}
Assume without loss of generality that $(M_\beta f)'(x)>0$, and let $z=x+r_x$. Then, by \eqref{eq:derivative d=1}, integration by parts and \eqref{char}
\begin{align}
    (M_\beta f)'(x) & \, =\, r_x^\beta \intav_{B_x} |f|'(y) \frac{(z-x)}{r_x} \dy \notag  \\
    & =\, \frac{r_x^{\beta}}{r_x} \left[ \intav_{B_x} |f|'(y) (z-y) \dy +  \intav_{B_x} |f|'(y) (y-x) \dy \right] \notag  \\
    & =\, \frac{r_x^{\beta}}{r_x} \left[ - |f(x-r_x)| + \intav_{B_x} |f(y)|  \dy - \beta \intav_{B_x} |f(y)| \dy \right]. \label{d=1 int by parts}
\end{align}
Thus, it suffices to show that
$$
 |f|_{B_x} -  |f(x-r_x)|\leq 2\int_{B_x}|f'(z)|\, \chi_{E_x}(z)\dz\,.
$$

As $f \in W^{1,1}(\R)$, $f$ is continuous. In particular, we can choose $z_0\in B_x$ such that $f(z_0)=|f|_{B_x}$. 
It is immediate from \eqref{d=1 int by parts} that $f(z_0)\geq f(x-r_x)$. We will analyse two different situation according to the relative size of $f(x-r_x)$ and $f(z_0)/2$. Note that $z_0\in E_x$.

{\it{Case 1:}} $f(x-r_x)\geq f(z_0)/2$. In this case we have that $x-r_x\in E_x$. By the continuity of $f$ there exist $z_1,z'_0\in B_x$ such that $f(z_1)=f(x-r_x)$ and $f(z'_0)=f(z_0)$ with $[z_1,z'_0]\subseteq E_x$ (or $[z'_0,z_1] \subseteq E_x$). Then
$$
f(z_0)-f(x-r_x)=f(z'_0)-f(z_1)\leq \int_{B_x}|f'(z)|\, \chi_{E_x}(z)\dz.
$$

{\it{Case 2:}} $f(x-r_x)< f(z_0)/2$. By the continuity of $f$ there exists $z_2,z'_0\in [x-r_x,z_0]$ such that $f(z_2)=f(z_0)/2$, $f(z'_0)=f(z_0)$ and  $[z_2,z'_0] \subseteq E_x$. Then
$$
f(z_0)-f(x-r_x)\leq f(z_0)=2(f(z'_0)-f(z_2)) \leq 2\int_{z_2}^{z'_0}|f'(z)|\, \chi_{E_x}(z)\dz.
$$

Combining both cases one obtains the desired result. 
\end{proof}

We can then proceed to the proof of \eqref{eq:endpoint sobolev d=1}, and therefore of Theorem \ref{thm:d=1}.

\begin{proof}[Proof of Theorem \ref{thm:d=1}]\label{subsec: Proof of thm 1.1}
By Lemma \ref{lemma:derivative d=1}, $M_\beta f$ is differentiable except on a set of measure zero. Thus, it suffices to show the bound \eqref{eq:endpoint sobolev d=1} on the set
\begin{equation*}
\Omega:=\{\,x \in \R\,:\, M_{\beta}f \ \text{is differentiable at}\ x \ \,\, \text{and}\, \,\ (M_{\beta}f)'(x)\not=0\,\}\,.
\end{equation*}
For each $x \in \Omega$, fix $B_x:=B(x,r_x)\in \mathcal{B}_x^\beta$ such that $r_x$ is the smallest possible radius. 

By Lemma \ref{lemma:derivative d=1} and Lemma \ref{lemma2 d=1}, one has
\begin{align*}
\int_{\Omega}  |(M_{\beta}f)'(x)|^q \dx\,&=\, \int_{\Omega}\bigg|r_x^{\beta}\intav_{B_x} |f|' (y) \dy\,\bigg|^q \dx\\
&= \int_{\Omega}\frac{r_x^{q\beta}}{2^{q-1}r_x^{q-1}}\bigg|\int_{B_x}  |f|' (y) \dy\,\bigg|^{q-1}\bigg|\intav_{B_x}  |f|'(y) \dy\,\bigg| \dx\,\\
&\leq  \,\frac{1}{2^{q-1}}\norm{ f'}_1^{q-1}\int_{\Omega}\bigg|\intav_{B_x}  |f|'(y) \dy\,\bigg| \dx\,\\
&\leq  \,\frac{4}{2^{q-1}}\norm{ f'}_1^{q-1}   \int_{\R}\,|f'(y)|\bigg(\int_{\Omega}\frac{\chi_{B_x}(y)\chi_{E_x}(y)}{|B_x|} \dx\,\bigg) \dy\,
\end{align*}
using that $q\beta=q-1\,$ and where $E_x$ is the set defined in Lemma \ref{lemma2 d=1}. 

We analyse the inner integral for a fixed $y\in\R$. We may assume that there exist $x_1, x_2 \in \R$ such that
\begin{equation*}
\chi_{B_{x_1}}(y)\chi_{E_{x_1}}(y)\not =0\qquad \text{ and }\qquad \chi_{B_{x_2}}(y)\chi_{E_{x_2}}(y)\not =0\,,
\end{equation*}
as otherwise the inner integral vanishes. In particular, this implies that $B_{x_1} \cap B_{x_2} \neq \emptyset$ and, by definition of $E_{x_1}$ and $E_{x_2}$,
\begin{align*}
\frac{1}{2}|f|_{B_{x_1}}\leq |f(y)|\leq 2|f|_{B_{x_1}}\qquad \text{ and } \qquad \frac{1}{2}|f|_{B_{x_2}}\leq |f(y)|\leq 2|f|_{B_{x_2}}\,.
\end{align*}
Therefore
\begin{equation*}
\frac{1}{4}|f|_{B_{x_1}}\leq |f|_{B_{x_2}}\leq 4|f|_{B_{x_1}}\,,
\end{equation*}
and by Proposition \ref{prop comparable radii} we obtain
\begin{equation*}
4^{-\frac{1}{\beta}} 3^{-\frac{1-\beta}{\beta}}\leq \frac{r_{x_1}}{r_{x_2}} \leq 4^{\frac{1}{\beta}} 3^{\frac{1-\beta}{\beta}}.
\end{equation*}
This means that there exists $x_0 \in \Omega_y:=\{ x \in \Omega  : \chi_{B_{x}}(y)\chi_{E_{x}}(y)\not =0 \}$ such that 

\begin{equation*}\label{eqc}
|x-y|\leq r_x\leq 4^{\frac{1}{\beta}} (3^{\frac{1-\beta}{\beta}}) r_{x_0}
\end{equation*}
and
\begin{equation*}|B_x|=2r_x\geq 2(4^{-\frac{1}{\beta}} 3^{-\frac{1-\beta}{\beta}})r_{x_0} = 4^{-\frac{1}{\beta}} 3^{-\frac{1-\beta}{\beta}}|B_{x_0}|\,
\end{equation*}
for all $x \in \Omega_y$. Finally, this implies that 
\begin{align*}
\int_{\Omega}\frac{\chi_{B_x}(y)\chi_{E_x}(y)}{|B_x|} \dx\,&\leq 4^{\frac{1}{\beta}} 3^{\frac{1-\beta}{\beta}} \int_{B(y,4^{\frac{1}{\beta}} 3^{\frac{1-\beta}{\beta}}r_{x_0})} \frac{\mathrm{d}x}{|B_{x_0}|}\\ \,&\leq\,4^{\frac{2}{\beta}} 3^{\frac{2(1-\beta)}{\beta}}
\end{align*} 
for all $y \in \R$. Altogether,
\begin{align*}
\int_{\R}  |(M_{\beta}f)'(x)|^q \dx
\leq&  \,\frac{4}{2^{q-1}}\norm{ f'}_1^{q-1}   \int_{\R}\,|f'(y)|\bigg(\int_{\Omega}\frac{\chi_{B_x}(y)\chi_{E_x}(y)}{|B_x|} \dx\,\bigg) \dy\,\\
\leq& \, 2^{-\frac{\beta}{1-\beta}} 4^{\frac{2}{\beta} + 1} 3^{\frac{2(1-\beta)}{\beta}}\|f'\|^{q}_1,
\end{align*}
as desired.
\end{proof}


\section{Relation between $\nabla M_\beta f$ and $\nabla \widetilde{M}_\beta f$}\label{sec:relation}

The goal of this section is to establish an inequality that relates $\nabla M_\beta f$ with its non-centered counterpart $\nabla \widetilde{M}_\beta f$. This inequality, which is shown in Lemma \ref{lemma radios grandes}, may be seen as the analogue at the derivative level of the trivial bound $M_\beta f \leq \widetilde{M}_\beta f$, and it is the first time that such a relation has been obtained at the derivative level. In particular, it will allow us to use some of the techniques in \cite{LM2017} to deal with the term $\nabla \widetilde{M}_\beta f$, although additional difficulties will arise.

To show the upcoming Lemma \ref{lemma radios grandes} we will need several auxiliary results. Some of them were already observed for $\widetilde{M}_\beta$ in \cite{LM2017}, whilst the ones concerning $M_\beta$ are new.

The next proposition was established in \cite[Proposition 2.5]{LM2017}, 
and will also be useful in our case.

\begin{proposition}\label{prop3}
 If $f\in W^{1,1}_{loc}(\R^d)$, $x\in\R^d$, $r>0$, then 
 \begin{equation*}
  \intav_{B(x,r)}\nabla |f|(y)\cdot(x-y)\,\dy\,=\,d\bigg[\intav_{B(x,r)}|f|-\,\intav_{\partial B(x,r)}|f|\bigg]\,.
 \end{equation*}
\end{proposition}

However, we will also need the following variant, which is more suitable in the study of the centered fractional maximal function.

\begin{proposition}\label{prop:centered spherical}
If $f \in W^{1,1}_{loc}(\R^d)$, $x \in \R^d$, $r>0$ and $z \in \partial B(x,r)$, then
\begin{align*}
    \intav_{B(x,r)} \nabla|f|(y)\cdot(z-y) \dy & = d\left[\intav_{B(x,r)}|f(y)| \dy-\intav_{\partial B(x,r)}|f(y)| \dy\right] \\
    &  \qquad  + \frac{d}{r^2} \intav_{\partial B(x,r)} |f(y)| (z- x) \cdot (y - x) \dy.
\end{align*}

\end{proposition}

\begin{proof}
By Gauss divergence theorem,
\begin{align*}
\int_{B(x,r)} & \nabla|f| (y)\cdot(z-y) \dy \\
&=\sum_{i=1}^{d}\int_{B(x,r)}\frac{\partial|f|}{\partial y_{i}}(y) \, (z_{i}-y_{i}) \dy\\
&=\sum_{i=1}^{d}\left(\int_{\partial B(x,r)}|f(y)|(z_{i}-y_{i}) \frac{(y_{i}-x_{i})}{|y-x|} \dy+\int_{B(x,r)}|f(y)| \dy\right)\\
&=-\int_{\partial B(x,r)}|f(y)|\sum_{i=1}^{d}\frac{|y_{i}-x_{i}|^{2}}{|y-x|}\dy\\
&\ \ \ \ \ \quad + \int_{\partial B(x,r)} |f(y)|\sum_{i=1}^d (z_i - x_i) \frac{(y_i - x_i)}{|y-x|} \dy +\sum_{i=1}^{d}\int_{B(z,r)}|f(y)| \dy\\
&=d\int_{B(x,r)}|f(y)| \dy-\int_{\partial B(x,r)}|f(y)||y-x| \dy\\
&\ \ \ \ \ \quad + \int_{\partial B(x,r)} |f(y)| (z - x) \cdot \frac{(y - x)}{|y-x|} \dy \\
&=d\left[\int_{B(x,r)}|f(y)| \dy-\frac{r}{d}\int_{\partial B(x,r)}|f(y)| \dy\right]\\
&\ \ \ \ \ \quad + \frac{1}{r} \int_{\partial B(x,r)} |f(y)| (z - x) \cdot(y - x) \dy\\
&=d\left[\int_{B(x,r)}|f(y)| \dy-\frac{r^{d}w_{d}}{r^{d-1}\sigma_{d}}\int_{\partial B(x,r)}|f(y)| \dy\right]\\
&\ \ \ \ \ \quad + \frac{d r^d \omega_d}{r^2 r^{d-1} \sigma_d} \int_{\partial B(x,r)} |f(y)| (z - x) \cdot (y - x)\dy \\
&=d\left[\int_{B(x,r)}|f(y)| \dy-|B(x,r)|\intav_{\partial B(x,r)}|f(y)| \dy\right] \\
&\ \ \ \ \ \quad+ \frac{d |B(x,r)|}{r^2} \intav_{\partial B(x,r)} |f(y)|\, (z- x) \cdot (y - x) \dy.
\end{align*}
The result follows dividing by $|B(x,r)|$.
\end{proof}

Propositions \ref{prop3} and \ref{prop:centered spherical} can be combined to yield the following.

\begin{corollary}\label{cor:centered}
Let $0 < \beta < d$. If $f \in W^{1,1}_{loc}(\R^d)$, $x \in \R^d$, $B(x,r) \in \CB_x^\beta$ and $z \in \partial B(x,r)$, then
\begin{equation*}
 \intav_{B(x,r)} \nabla|f|(y)\cdot(z-y) \dy \,  \leq \,  \frac{d^2}{\beta} \Big[ \intav_{B(x,r)} |f| - \intav_{\partial B(x,r)} |f| \Big].
\end{equation*}

\end{corollary}

\begin{proof}
Note that by the Cauchy-Schwarz inequality,
$$
\frac{d}{r^2} \intav_{\partial B(x,r)} |f(y)| (z- x) \cdot (y - x) \dy \leq d \intav_{\partial B(x,r)} |f(y)| \dy.
$$
Therefore, Proposition \ref{prop:centered spherical}, \eqref{char} and Proposition \ref{prop3} yield
\begin{align*}
 \intav_{B(x,r)} \nabla|f|(y)\cdot(z-y) \dy  \, & \leq \, d \intav_{B(x,r)}|f(y)| \dy \\
 & =\, \frac{d}{\beta} \intav_{B(x,r)} \nabla |f|(y) \cdot (x-y) \dy  \\
 & =\, \frac{d^2}{\beta} \Big[ \intav_{B(x,r)} |f| - \intav_{\partial B(x,r)} |f| \Big],
\end{align*}
as desired.
\end{proof}

\begin{lemma}\label{lemma1}
Let $0<\beta<d$. Suppose that $f\in W^{1,1}_{loc}(\R^d)$ is radial, and let $B=B(x,r)\in \mathcal{B}^{\beta}_x$ for some $x\in\R^d$. Then
\begin{equation*}
\bigg|\intav_{B}\nabla |f|(y)\dy\,\bigg|\,\leq \,\frac{d^2}{\beta r}\bigg[\big(1-\beta^2/d^2)\intav_{B}|f(y)|\dy\,-\intav_{\partial B}|f(y)| \dy\,\bigg]\,.
\end{equation*}
\end{lemma}
\begin{proof}
Suppose that $B=B(x,r)$ and $z \in \partial B(x,r)$ such that $z= x \pm r \frac{x}{|x|}$, where the sign is chosen according to the direction of $\nabla M_\beta f(x)$. By \eqref{char} and Corollary \ref{cor:centered}, it follows that
\begin{align*}
\bigg|\intav_{B}\nabla |f|(y)\dy\,\bigg|\, & =\,\intav_{B}\nabla|f|(y)\cdot\bigg(\frac{z-x}{r}\bigg)\dy\,\\
&= \,\frac{1}{r}\bigg[\intav_{B}\nabla |f|(y)\cdot(z-y)\dy\,+\intav_{B} \nabla |f|(y)\cdot(y-x)\dy\,\bigg]\,\\
&=\,\frac{1}{r}\bigg[\intav_{B} \nabla |f|(y)\cdot(z-y)\dy\,-\beta\intav_{B}|f(y)|\dy\,\bigg]\\
&\leq \, \frac{1}{r}\bigg[\frac{d^2}{\beta}\bigg[\intav_{B}|f|-\,\intav_{\partial B}|f|\bigg]-\beta\intav_{B}|f(y)|\dy\,\bigg]\\
&= \,\frac{d^2}{\beta r}\bigg[\big(1-\beta^2/d^2)\intav_{B}|f(y)|\dy\,-\intav_{\partial B}|f(y)| \dy\,\bigg]\,.
\end{align*}
\end{proof}

The following two lemmas contain the key estimates for the proof of the main theorem. Of these, the next one is the novel one, as it includes the aforementioned relation between $M_\beta f$ and $\widetilde{M}_\beta f$ at the derivative level.
\begin{lemma}\label{lemma radios grandes}
Let $0< \beta < d$. Suppose that  $f\in W^{1,1}_{loc}(\R^d)$ is radial and $B\in \mathcal{B}^{\beta}_x$, $\widetilde{B} \in \widetilde{\mathcal{B}}^\beta_x$ for some $x\in\R^d\setminus\{0\}$.
\begin{enumerate}[i)]
    \item If $\nabla M_\beta f (x) \cdot x \leq 0$, then
    \begin{equation*}
        \bigg|\intav_{B(x,r)} \nabla |f|(y)\dy\,\bigg|\,\leq\,\intav_{B(x,r)}|\nabla |f|(y)|\frac{|y|}{|x|} \dy\,.
    \end{equation*}
    \item If $\nabla M_\beta f (x) \cdot x > 0$, then
    \begin{equation*}
        |\nabla M_\beta f(x)|\,\leq\,  r^\beta \intav_{B(x,r)} \nabla |f|(y)\cdot \frac{y}{|x|}\dy  -  \widetilde{r}^\beta \intav_{\widetilde{B}} \nabla |f| (y)  \cdot  \frac{y}{|x|} \dy + \nabla \widetilde{M}_\beta f (x)  \cdot  \frac{x}{|x|},
    \end{equation*}
    where $r$ and $\widetilde r$ denote the radii of $B$ and $\widetilde{B}$ respectively.
\end{enumerate}

\begin{proof}
If $|\nabla M_\beta f(x)|=0$, the claim is trivial. If $\nabla M_\beta f(x) \cdot x < 0$, iii) in Lemma \ref{basic} yields that
\begin{equation*}
\frac{\intav_{B(x,r)}\nabla |f|(y)\,\dy\,}{\bigg|\intav_{B(x,r)}\nabla |f|(y)\,\dy\,\bigg|}\,=\,\frac{-x}{|x|}\,,
\end{equation*}
and by ii) in Lemma \ref{basic},
\begin{align*}
\bigg|\intav_{B(x,r)}\nabla |f|(y) \dy\,\bigg|\, & =\,\intav_{B(x,r)} \nabla |f|(y)\cdot \frac{-x}{|x|}\,\dy\,\\
& =\,\intav_{B(x,r)} \nabla |f|(y)\cdot \frac{-y}{|x|}\,\dy\,-\frac{\beta}{|x|}\intav_{B(x,r)}|f(y)|\,\dy\\
&\leq \, \intav_{B(x,r)}|\nabla |f|(y)|\frac{|y|}{|x|}\,.
\end{align*}
If $\nabla M_\beta f (x) \cdot x >0$, we cannot proceed as before as the term $\frac{\beta}{|x|}\intav_{B(x,r)} | f (y)| \dy$ is actively contributing. Instead, we control it by the non-centered maximal function $\widetilde{M}_\beta f (x)$. More precisely, as $M_\beta f (x) \leq \widetilde{M}_\beta f (x)$, an immediate consequence of Corollary \ref{coro good eq} is
\begin{align*}
    |\nabla M_\beta f (x)|\,\leq \, r^\beta \intav_{B(x,r)} \nabla |f|(y)\cdot \frac{y}{|x|} \dy\,+\frac{\beta}{|x|} \widetilde{M}_\beta f (x).\,
\end{align*}
Let $\widetilde{B}=B(z, \widetilde{r}) \in \widetilde{\mathcal{B}}^\beta_x$. By Corollary \ref{coro good eq} for $\widetilde{M}_\beta$,
$$
\beta \widetilde{M}_\beta f (x) = \widetilde{r}^\beta \intav_{\widetilde{B}} \nabla |f| (y) \cdot (x-y) \dy,
$$
so altogether,
\begin{align*}
|\nabla M_\beta f (x)|\, &\leq \, r^\beta \intav_{B(x,r)} \nabla |f|(y)\cdot \frac{y}{|x|}\dy\,+ \widetilde{r}^\beta \intav_{\widetilde{B}} \nabla |f| (y)  \cdot  \frac{x}{|x|} \dy \\
&\ \ \ \ \ - \, \widetilde{r}^\beta \intav_{\widetilde{B}} \nabla |f| (y)  \cdot  \frac{y}{|x|} \dy \\
& 
= r^\beta \intav_{B(x,r)} \nabla |f|(y)\cdot \frac{y}{|x|}\dy \, - \, \widetilde{r}^\beta \intav_{\widetilde{B}} \nabla |f| (y)  \cdot  \frac{y}{|x|} \dy \\
&\ \ \ \ \ + \nabla \widetilde{M}_\beta f (x)  \cdot  \frac{x}{|x|},
\end{align*}
as desired.
\end{proof}
\end{lemma}

The above lemma will be used when $M_\beta f (x) \neq M_\beta^I f(x)$. When those two maximal functions are equal, we use instead the following lemma, which is a minor modification of its non-centered counterpart \cite[Lemma 2.10]{LM2017}.

\begin{lemma}\label{lemma2}
Suppose that $f\in W^{1,1}_{loc}(\R^d)$ is radial, $0<\beta<d$, $B_x=B(x,r_x)\in \mathcal{B}^{\beta}_x$ for some $x\in\R^d$ with $r_x\leq\frac{|x|}{4}\,$, and
\begin{equation*}
E_x:=\{z \in 2B_x\,:\,\frac{1}{2}|f|_{B_x}\leq |f(z)|\leq 2|f|_{B_x}\}\,.
\end{equation*}
 Then
\begin{equation*}
\bigg|\intav_{B_x} \nabla |f|(y) \dy\,\bigg|\leq C(d,\beta)\intav_{2B_x}|\nabla f(z)|\, \chi_{E_x}(z)\dz\,.
\end{equation*}
\end{lemma}

Note that this lemma is a radial higher-dimensional counterpart of Lemma \ref{lemma2 d=1} under the additional assumption $r_x \leq |x|/4$. This condition, together with Proposition \ref{radial1}, reduces the proof of the lemma to the one-dimensional case, provided the integration by parts argument is replaced by the estimate in Lemma \ref{lemma1}. Proposition \ref{radial1} is an elementary observation, but nevertheless crucial in order to extend the one-dimensional analysis to that of $M_\beta^I$ over radial functions.

\begin{proposition}\label{radial1}
Suppose that  $f\in L^{1}_{loc}(\R^d)$ satisfies $f(x)=F(|x|)$, $F:(0,\infty)\to[0,\infty)$, $B:=B(z,r)\subset B(0,2|z|)\setminus B(0,\frac{1}{2}|z|)\,$, and 
$a:= |z|-r$, $b:=|z|+r$. 
Then 
it holds that
\begin{equation*}
\intav_{[a,b]}F(t)\dt\,\leq C(d)\intav_{B(z,2r)}f(y)\dy\,.
\end{equation*}
\end{proposition}

\section{Boundedness for $d>1$}\label{sec:main theorem}

This section is devoted to the proof of the estimate \eqref{eq:sobolev bound}. We will examine the different possible cases that arise depending on: the size of good radii $r_x$ relative to $|x|$, the direction of $\nabla M_\beta f(x)$ and the size of auxiliary good radii $\widetilde{r}_x$ for the non-centered $\widetilde{M}_\beta f(x)$. A sketch of the proof is provided in Figure 1 for the reader's convenience.

\begin{figure}[h]\label{arbol}

\tikzstyle{level 1}=[level distance=3.5cm, sibling distance=3.5cm]
\tikzstyle{level 2}=[level distance=3.5cm, sibling distance=2cm]
\tikzstyle{level 3}=[level distance=3.5cm, sibling distance=2cm]

\tikzstyle{bag} = [text width=4em, text centered]
\tikzstyle{end} = [circle, minimum width=3pt,fill, inner sep=0pt]

\begin{tikzpicture}[grow=right, sloped]
\node[bag] {$\nabla M_\beta f(x)$}
    child {
        node[bag] {$\Omega_2$: Lemma \ref{lemma radios grandes}}        
            child {
                node[bag] {$\Omega_2^+$}
                child{
                    node[end, label=right:
                        { $\approx$ $\Omega_2^-$ }] {}
                    edge from parent
                    node[below] {$\widetilde{r}_x \geq |x|/4$}
                }
                child{
                    node[end, label=right:
                        {$ \approx \Omega_2^- + \Omega_1$}] {}
                    edge from parent
                    node[above] {$\widetilde{r}_x \leq |x|/4$}
                    node[below]  {}}
                    edge from parent       
            node[below]  {$\nabla M_\beta f(x) \cdot x > 0$}
            }
            child {
                node[end, label=right:
                    {$\Omega_2^-$}] {}
                edge from parent
                node[above] {$\nabla M_\beta f(x) \cdot x \leq 0$}
                node[below]  {}
            }
            edge from parent 
            node[below] {$r_x \geq |x|/4$}
    }
    child {
        node[end, label=right: {$\Omega_1$: Lemma \ref{lemma2}}]{}
        edge from parent         
            node[above] {$r_x \leq |x|/4$}
    };
\end{tikzpicture}
\caption{Scheme of the proof depending on the different cases. The notation $\approx$ means that the analysis is similar to that performed in such cases.}
\end{figure}
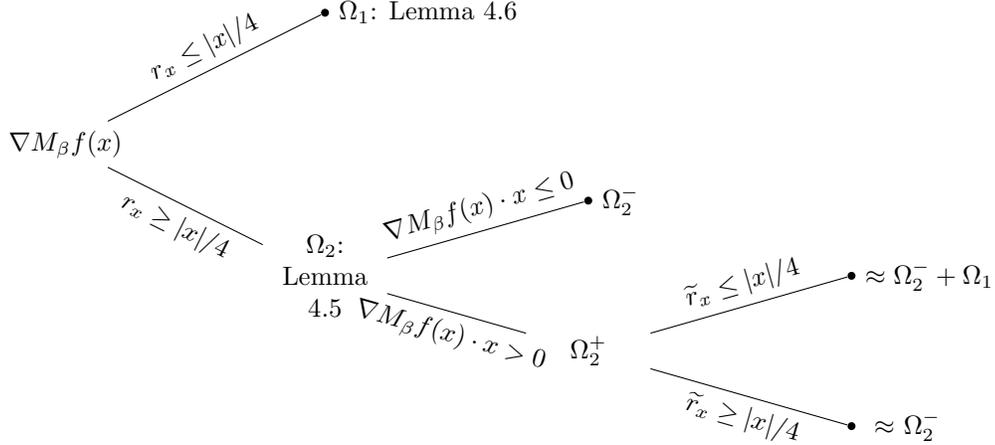

Let's turn into the proof. By Lemma \ref{basic}, $M_\beta f$ and $\widetilde M_{\beta}f$ are differentiable except on a set of measure zero. Thus, it suffices to show the bound \eqref{eq:sobolev bound} on the set
\begin{equation*}
\Omega:=\{\,x \in \R^{d}\,:\, M_{\beta}f\  \ \text{and}\ \ \widetilde M_{\beta}f \ \text{are differentiable at}\ x \ \,\, \text{and}\ \,\, \nabla M_{\beta}f(x)\not=0\,\}\,.
\end{equation*}

For each $x \in \Omega$, fix $B_x:=B(x,r_x)\in \mathcal{B}_x^\beta$ such that $r_x$ is the smallest possible good radius. Define the sets
\begin{equation*}
\Omega_1:=\{x\in \Omega\,:\, r_x\leq \frac{|x|}{4}\}\qquad
\text{and} \qquad
\Omega_2:=\{x\in \Omega\,:\, r_x> \frac{|x|}{4}\}.
\end{equation*}
Then we can estimate
\begin{align*}
\int_{\R^d}  |\nabla  M_{\beta}f(x)|^q \dx & = \int_{\Omega}\bigg|r_x^{\beta}\intav_{B_x}\nabla |f| (y) \dy\,\bigg|^q  \dx\\
&= \int_{\Omega}\!\frac{r_x^{q\beta}}{(\omega_d)^{q-1}r_x^{d(q-1)}}\bigg|\int_{B_x}\! \nabla |f| (y)\dy\bigg|^{q-1}\bigg|\intav_{B_x}\!\nabla |f|(y)\dy \bigg| \dx\\
&\lesssim \norm{\nabla f}_1^{q-1}\int_{\Omega}\bigg|\intav_{B_x}\nabla |f|(y)\dy\,\bigg|\dx\,\\
&=\norm{\nabla f}_1^{q-1}\sum_{i=1}^2\int_{\Omega_i}\bigg|\intav_{B_x}\nabla |f|(y)\dy\,\bigg|\dx\,,
\end{align*}
where we used the fact $q\beta=d(q-1)\,$. It suffices to show that 
\begin{equation*}
\int_{\Omega_i}\bigg|\intav_{B_x} \nabla |f| (y)\dy\,\bigg|\dx\,\lesssim \norm{\nabla f}_1\,\quad \text{ for }\, \, i=1,2\,.
\end{equation*}

In the case of $\Omega_1$, the bound follows using Lemma \ref{lemma2} and the same scheme as in the proof of the one-dimensional case given in Section \ref{subsec: Proof of thm 1.1}. This case is then analogous to its non-centered counterpart, already established in \cite{LM2017}. The details are ommited.

The case of $\Omega_2$ is where the analysis substantially differs from that of $\widetilde{M}_\beta$. We consider two further subcases. Define
$$
\Omega_2^+:= \{x \in \Omega_2 : \nabla M_\beta f (x) \, \cdot \, x > 0\}
$$
and
$$
\Omega_2^-:= \{x \in \Omega_2 : \nabla M_\beta f (x) \, \cdot \, x \leq 0\}.
$$
\underline{\textit{The case $\Omega_2^-$}}: By i) in Lemma \ref{lemma radios grandes} one has
\begin{align*}
    \int_{\Omega_2^-} \Big| \intav_{B_x} \nabla |f| (y) \dy \Big| \dx 
    \leq  \int_{\Omega_2^-} \intav_{B_x} |\nabla |f| (y)| \frac{|y|}{|x|} \dy \dx.
\end{align*}
Note that $|x| < 4r_x$ and $|y| \leq |x| + 2r_x \leq 6r_x$. Then
\begin{align*}
    \int_{\Omega_2^-} \frac{1}{|B_x|}  &\int_{B_x \cap \{ |y| > |x| \}} |\nabla |f| (y)| \frac{|y|}{|x|} \dy \dx\\
    &  \lesssim  \int_{\Omega_2^-} \frac{1}{|x|} \int_{B_x \cap \{ |y| > |x| \}} |\nabla|f| (y)| \frac{1}{|y|^{d-1}} \dy \dx \\
    &  =  \int_{\R^d} \frac{|\nabla |f| (y)|}{|y|^{d-1}} \int_{0}^{|y|} \frac{1}{|x|} \dx \dy \\
    &  \lesssim \int_{\R^d} |\nabla f(y)| \dy
\end{align*}
simply by Fubini and integrating in the $x$-variable. Similarly,
\begin{align*}
    \int_{\Omega_2^-} \frac{1}{|B_x|}& \int_{B_x \cap \{ |y| \leq |x| \}} |\nabla |f| (y)| \frac{|y|}{|x|} \dy \dx\\
    & \lesssim \int_{\Omega_2^-} \int_{B_x \cap \{ |y| \leq |x| \}} |\nabla |f| (y)| \frac{|y|}{|x|^{d+1}} \dy \dx \\
    & =  \int_{\R^d} |\nabla |f| (y)| |y| \int_{|y|}^{\infty} \frac{1}{|x|^{d+1}} \dx \dy \\
    & \lesssim \int_{\R^d} |\nabla f(y)| \dy
\end{align*}
which follows again by Fubini and integration in $x$. Putting the above estimates together one has that
$$
\int_{\Omega_2^-} \Big|\intav_{B_x} \nabla |f| (y) \dy \Big| \dx \lesssim \| \nabla f \|_1,
$$
as desired.

\underline{\textit{The case $\Omega_2^+$}}: 
By ii) in Lemma \ref{lemma radios grandes}, we have that
\begin{equation}\label{eq:crucial}
|\nabla  M_\beta f(x)| \leq  r_x^\beta \intav_{B_x } \nabla |f|(y)\cdot \frac{y}{|x|}\dy \, - \, \widetilde{r}_x^\beta \intav_{\widetilde{B}_x} \nabla |f| (y)  \cdot  \frac{y}{|x|} \dy \,+ \nabla \widetilde{M}_\beta f (x)  \cdot  \frac{x}{|x|}
\end{equation}
where $\widetilde{B}_x=B(z,\widetilde{r}_x) \in \widetilde{\CB}_x^\beta$ is chosen so that $\widetilde{r}_x$ is the smallest possible good radius. We will argue differently depending on the size of $\widetilde{r}_x$.

If $\widetilde r_x \leq \frac{|x|}{4}$, one may use item iv) in Lemma \ref{basic} in the above estimate  to obtain 
\begin{align*}
|\nabla & M_\beta f(x)|\\
&\leq r_x^\beta \intav_{B_x} \nabla |f|(y)\cdot \frac{y}{|x|}\dy \, +\, \widetilde{r}_x^\beta \intav_{\widetilde{B}_x} \nabla |f| (y)  \cdot  \frac{(z-y) + (x-z)}{|x|} \dy \,\\
&= r_x^\beta \intav_{B_x} \nabla |f|(y)\cdot \frac{y}{|x|}\dy \, + \, \frac{\widetilde{r}_x^\beta}{|x|} \intav_{\widetilde{B}_x} \nabla |f| (y)  \cdot  (z-y) \dy  \, +\, \nabla \widetilde{M}_\beta f (x)  \cdot  \frac{x-z}{|x|}  \\
&= r_x^\beta \intav_{B_x} \nabla |f|(y)\cdot \frac{y}{|x|}\dy \, +\, \frac{\widetilde{r}_x^\beta}{|x|} \intav_{\widetilde{B}_x} \nabla |f| (y)  \cdot  (z-y) \dy-\frac{|\nabla \widetilde M_{\beta}f(x)|\widetilde r_x}{|x|}  \\
&= r_x^\beta \intav_{B_x} \nabla |f|(y)\cdot \frac{y}{|x|}\dy \, +\, \frac{d \,  \widetilde{r}_x^\beta}{|x|} \left[\intav_{\widetilde{B}_x}  |f| -\intav_{\partial \widetilde B_x}|f|\right]-\frac{|\nabla \widetilde M_{\beta}f(x)|\widetilde r_x}{|x|},
\end{align*}
where the last equality follows from Proposition \ref{prop3}. Diving the above inequality by $r_x^\beta$ and dropping the last term one has,
$$
\Big| \intav_{B_x} \nabla |f| (y) \dy \Big| \leq \intav_{B_x} \nabla |f|(y)\cdot \frac{y}{|x|}\,\dy \, +\, \Big(\frac{ \widetilde{r}_x}{r_x} \Big)^\beta \frac{d}{|x|} \left[\intav_{\widetilde{B}_x}  |f| -\intav_{\partial \widetilde B_x}|f|\right].
$$
As $\widetilde{r}_x\leq |x|/4$, one has $(\widetilde{r}_x/r_x) \leq 1$ and
\begin{align*}
    \int_{\Omega_2^+} \Big| \intav_{B_x} \nabla |f| (y) \dy \Big| \dx 
    & \leq \int_{\Omega_2^+} \intav_{B_x} |\nabla |f| (y)| \frac{|y|}{|x|} \dy\dx\\
    &\ \ \ \ \ + d  \int_{\Omega_2^+}\frac{1}{|x|} \left[\intav_{\widetilde{B}_x}  |f (y)|\dy-\intav_{\partial \widetilde B_x}|f(y)|\dy\right]\dx\\
    & = I+ d \, II.
\end{align*}
Then, to estimate $I$ we can proceed as in the estimate for $\Omega_2^-$ and to estimate $II$ we can proceed as in the estimate for $\Omega_1$ through Lemma \ref{lemma2}, bounding $1/|x| \leq 1/|\widetilde{r}|$. This yields the desired bound on $\Omega_2^+$ when $\widetilde r_x \leq |x|/4$.

Finally, we assume $\widetilde r_x>\frac{|x|}{4}$. Consider first the case $\nabla \widetilde{M}_\beta f(x)=0$. Dividing \eqref{eq:crucial} by $r_{x}^\beta$ one has
\begin{align*}
&\int_{\Omega_2^+ \cap \{ x:  \nabla \widetilde M_\beta f(x) =0\}} 
\Big| \intav_{B_x} \nabla |f| (y) \dy \Big| \dx \\
& \qquad \qquad  \leq  \int_{\Omega_2^+}  \intav_{B_x } |\nabla |f|(y)|  \frac{|y|}{|x|}\dy  \dx \, + \,  \int_{\Omega_2^+} \frac{\widetilde{r}_x^\beta}{r_x^\beta} \intav_{\widetilde{B}_x} |\nabla |f| (y)|  \frac{|y|}{|x|} \dy \dx
\end{align*}
The first term can be bounded by $\| \nabla f \|_1$ as in $\Omega_2^-$. For the second term, note that if $y \in \widetilde{B}_x$, $|y| \leq 2\widetilde{r_x} + |x| \leq 6 \widetilde{r}_x$. This is the same situation as in the first term, except that we have the additional term $(\widetilde{r_x} / r_x)$, which cannot be ensured to be less than 1. However, it can essentially be treated in the same way using Fubini's theorem and the bounds $r_x \geq |x|/4$, $\widetilde r_x \geq |x|/4$ and $\widetilde{r}_x \geq |y|/6$,
\begin{align*}
\int_{\Omega_2^+} \frac{\widetilde r_x^\beta}{r_x^\beta} \intav_{\widetilde B_x} |\nabla |f| (y)| \frac{|y|}{|x|} \dy\dx 
&\lesssim \int_{\R^{d}}|\nabla |f|(y)||y|\int_{\{|x|\geq |y|\}}\frac{1}{|x|^{d+1}}\dx\dy\\ 
&\ \ \ \ \ +\int_{\R^{d}}\frac{|\nabla |f|(y)|}{|y|^{d-1-\beta}}\int_{\{|x|\leq |y|\}}\frac{1}{|x|^{1+\beta}}\dx\dy\\
&\lesssim\int_{\R^{d}}|\nabla f(y)|\dy.
\end{align*}
Assume next $|\nabla \widetilde M_\beta f(x)| \neq 0$. By iv) in Lemma \ref{basic} one has that $x \in \partial \widetilde B_x$, and moreover, by radiality and the non-centeredness it follows that $z=c_x x$ with either $c_x<1$ or $c_x > 1$. 

If $c_x<1$, as a consequence of iv) in Lemma \ref{basic}, 
we have that $\nabla \widetilde{M}_\beta f (x) \cdot x < 0$. Thus, that term can be dropped in \eqref{eq:crucial} and the same analysis as in $\nabla \widetilde M_\beta f(x)=0$ yields the estimate - in fact, the situation is even simpler as $\widetilde B_x \subseteq B(0, |x|)$.

If $c_x>1$, the term $\nabla \widetilde M_\beta f (x) \cdot x$ is actively contributing, and
\begin{align*}
    \int_{\Omega_2^+ \cap \{ x: c_x > 1\}} \Big| \intav_{B_x} \nabla |f| (y) \dy \Big| \dx \lesssim 
    &  \int_{\Omega_2^+ \cap \{ x: c_x > 1\}} \intav_{B_x} |\nabla |f| (y)| \frac{|y|}{|x|} \dy\dx\\
    &\ \ \ \ \ +   \int_{\Omega_2^+ \cap \{ x: c_x < 1\}} \frac{\widetilde r^{\beta}_x}{r^{\beta}_x}\intav_{\widetilde B_x} |\nabla |f| (y)| \frac{|y|}{|x|} \dy\dx\\
    &\ \ \ \ \ +   \int_{E^+_2\cap\{x: c_x>1\}}\frac{\nabla \widetilde M_{\beta}f(x)\cdot x}{|x|r^{\beta}_x}\dx.
\end{align*}
We can estimate the first two terms as in the case $|\nabla \widetilde M_\beta f (x)|=0$. The third term also follows with a similar argument, noting that
\begin{align*}
   \int_{\Omega^+_2\cap\{x:c_x>1\}}\frac{\nabla \widetilde M_{\beta}f(x)\cdot x}{|x|r^{\beta}_x}\dx&\lesssim \int_{\Omega^+_2\cap\{x: c_x>1\}}\frac{1}{r^{\beta}_x\widetilde r^{d-\beta}_{x}}\int_{\widetilde B_x}|\nabla|f|(y)|\dy\dx\\
   &\lesssim \int_{\R^{d}}\int_{\widetilde B_x\cap\{y:|y|\geq|x|\}}\frac{1}{|x|^{\beta}}\frac{|\nabla|f|(y)|}{|y|^{d-\beta}}\dy\dx\\
   &= \int_{\R^{d}}\frac{|\nabla|f|(y)|}{|y|^{d-\beta}}\int_{\{|x|\leq|y|\}}\frac{1}{|x|^{\beta}}\dx\dy\\
   &\lesssim \int_{\R^d}|\nabla|f|(y)|\dy.
\end{align*}
This yields the desired bound when $\widetilde{r}_x > |x|/4$ and concludes the proof of the endpoint Sobolev bound \eqref{eq:sobolev bound}.

\section{Continuity for $d>1$}\label{sec:continuity}

In order to conclude the proof of Theorem \ref{Main Theorem}, it remains to show that the map $f \mapsto |\nabla M_\beta f|$ is continuous from $W^{1,1}_{\mathrm{rad}}(\R^d)$ to $L^q(\R^d)$. As mentioned in the Introduction, the proof follows the strategy used by the authors \cite{BM2019} in the non-centered case, together with a new idea recently introduced by González--Riquelme \cite{GR} that allows to obtain smallness of $\nabla M_\beta f_j$ inside a small ball around the origin. In what follows we put this strategy in action; see also \cite[Theorem 25]{GR} for a similar approach.

For any radial function $f \in W^{1,1}(\R^d)$ and any sequence of radial functions $\{f_j\}_{j \in \N}$ in $W^{1,1}(\R^d)$ such that $\| f_j - f\|_{W^{1,1}(\R^d)} \to 0$ as $j \to \infty$, we want to show that
\begin{equation}\label{eq:goal}
\| \nabla M_\beta f_j - \nabla M_\beta f \|_{L^{d/(d-\beta)}(\R^d)} \to 0 \quad \textrm{as $j \to \infty$}.
\end{equation}

We first review some auxiliary results that were established in the previous work by the authors \cite{BM2019}.

\subsection{Preliminaries}\label{sec:preliminaries}

Given $f \in W^{1,1}(\R^d)$ and $\{f_j\}_{j \in \N} \subset W^{1,1}(\R^d)$ the associated families of good balls, defined in \eqref{def:good balls}, are simply denoted by $\CB_{x}^\beta$ and $\CB_{x,j}^\beta$ respectively. The families of good radii are denoted by $\CR_{x}^\beta$ and $\CR_{x,j}^\beta$.




\subsubsection{A Brézis--Lieb type reduction}\label{ae+BrezisLieb}

The classical Brézis--Lieb lemma \cite{BL1976} reduces the proof of \eqref{eq:goal} to showing that
$$
\int_{\R^d} |\nabla M_\beta f_j|^{\frac{d}{d-\beta}} \to \int_{\R^d} |\nabla M_\beta f|^{\frac{d}{d-\beta}} \quad \textrm{as $j \to \infty$}
$$
provided the almost everywhere convergence
\begin{equation}\label{eq:a.e. derivatives}
\nabla M_\beta f_j (x) \to \nabla M_\beta f(x) \quad \textrm{a.e. \:\: as $j \to \infty$}
\end{equation}
holds.




\subsubsection{Almost everywhere convergence of the derivatives}
The $\text{a.e.}$ convergence \eqref{eq:a.e. derivatives} was established in \cite{BM2019} provided the representation of the derivative \eqref{Luiro formula} holds.

\begin{lemma}[Lemma 2.4 \cite{BM2019}]\label{lemma:ae convergence derivatives}
Let $f \in W^{1,1}(\R^d)$ be a radial function and $\{f_j\}_{j \in \N} \subset W^{1,1}(\R^d)$ be a sequence of radial functions such that $\| f_j - f \|_{W^{1,1}(\R^d)} \to 0$ as $j \to \infty$. Then
\begin{equation*}
M_\beta f_j(x)  \to M_\beta f(x)\ \ \  \text{a.e.} \quad as \:\:j \to \infty,
\end{equation*}
and
\begin{equation*}
\nabla M_\beta f_j (x)  \to \nabla M_\beta f(x)  \ \ \ \text{a.e.} \quad as \:\:j \to \infty.
\end{equation*}
Moreover,
\begin{equation*}
 \nabla M^{I}_\beta f_j(x)  \to \nabla M^{I}_\beta f(x) \ \ \  \text{a.e} \quad as \:\:j \to \infty.
\end{equation*}
\end{lemma}




\subsubsection{A functional analytic convergence lemma}

In view of the representation of the derivative of $M_\beta$ in \eqref{Luiro formula}, it is useful to note that convergence of $f_j$ to $f$ in $W^{1,1}(\R^d)$ implies the convergence of their modulus. A proof of this functional analytic fact can be founded in \cite{BM2019}.

\begin{lemma}[Lemma 2.3 \cite{BM2019}]\label{lemma:convergence modulus in W11}
Let $f \in W^{1,1}(\R^d)$ and $\{f_j\}_{j \in \N} \subset W^{1,1}(\R^d)$ be such that $\| f_j - f \|_{W^{1,1}(\R^d)} \to 0$ as $j \to \infty$. Then $\| |f_j| - |f| \|_{W^{1,1}(\R^d)} \to 0$ as $j \to \infty$.
\end{lemma}




\subsubsection{A classical convergence result}

Finally, the following classical variant of the dominated convergence theorem will be also used in establishing \eqref{eq:goal}.

\begin{theorem}[Generalised dominated convergence theorem]\label{thm:gdct}
Let $1 \leq p < \infty$ $f, g \in L^p(\R^d)$  and $\{f_j\}_{j \in \N}$ and $\{g_j\}_{j \in \N}$ be sequences of functions on $L^p(\R^d)$ such that
\begin{enumerate}[i)]
    \item  $|f_j(x)| \leq |g_j(x)|$  a.e.,
    \item  $f_j(x) \to f(x)$ and $g_j(x) \to g(x)$ a.e. as $j \to \infty$,
    \item  $\| g_j - g \|_{L^p(\R^d)} \to 0$.
\end{enumerate}
Then $\| f_j - f \|_{L^p(\R^d)} \to 0$.
\end{theorem}
This follows as a consequence of Fatou's lemma; see for instance \cite[Chapter 4, Theorem 19]{royden2010real}.




\subsection{Proof of the continuity in Theorem \ref{Main Theorem}}

As it was shown by the authors in \cite{BM2019}, it suffices to show that the convergence \eqref{eq:goal} holds in any large compact set $K$. This is thanks to the following. 

\begin{proposition}[Proposition 4.10 \cite{BM2019}]\label{prop:smallness}
Let $0	<  \beta < d$, $f \in W^{1,1}(\R^d)$ and $\{f_j\}_{j \in \N} \subset W^{1,1}(\R^d)$ such that $\| f_j - f \|_{W^{1,1}(\R^d)} \to 0$. Then, for any $\varepsilon>0$ there exist a compact set $K$ and $j_\varepsilon>0$ such that
$$
\| \nabla M_\beta f_j - \nabla M_\beta f \|_{L^q( ( 3K)^c )} < \varepsilon
$$
for all $j \geq j_\varepsilon$, where $q=d/(d-\beta)$.
\end{proposition}

Our goal then is to establish the following.

\begin{proposition}\label{prop:convergence compact}
Let $0<\beta < 1$, $f \in W^{1,1}(\R^d)$ and $\{f_j\}_{j \in \N} \subset W^{1,1}(\R^d)$ be radial functions such that $\| f_j - f \|_{W^{1,1}(\R^d)} \to 0$. Then, for any compact set $K=\bar B(0,b)$,
\begin{equation*}\label{goal in compact}
\| \nabla M_\beta f_j - \nabla M_\beta f \|_{L^q(K)} \to 0 \qquad \textrm{as \: $j \to \infty$,}
\end{equation*}
where $q=d/(d-\beta)$.
\end{proposition}

Note that the convergence $\| \nabla M_\beta f_j - \nabla M_\beta f \|_{L^q(\R^d)} \to 0$ trivially follows combining the above propositions.

In order to prove Proposition \ref{prop:convergence compact}, write $K=B(0,a) \cup A(a,b)$, where $A(a,b)$ denotes the annulus $A(a,b):= \{ x \in \R^d : a \leq |x| \leq b\}$. We study the convergence on each region independently.

\begin{lemma}\label{anillo}
Let $0<\beta < 1$, $f \in W^{1,1}(\R^d)$ and $\{f_j\}_{j \in \N} \subset W^{1,1}(\R^d)$ be radial functions such that $\| f_j - f \|_{W^{1,1}(\R^d)} \to 0$. Then, for any $0<a<b<\infty$,
\begin{equation*}
\| \nabla M^{}_\beta f_j - \nabla M^{}_\beta f \|_{L^q(A(a,b))} \to 0 \qquad \textrm{as \: $j \to \infty$,}
\end{equation*}
where $q=d/(d-\beta)$.
\end{lemma}

\begin{proof}
Set $f_0=f$, and let $E_j$ be the set of measure zero for which Lemma \ref{basic}  fails for $f_j$. The set $E:=\cup_{j\geq 0} E_j$ continues to have measure zero. Moreover, let $F$ denote the set of measure zero for which Lemma \ref{lemma:ae convergence derivatives} fails. Note that for all $x \in A(a,b)\setminus{E \cup F}$,
$$
|\nabla M_\beta f_j (x)| \leq |\nabla M_\beta^I f_j (x)| + \frac{1}{ \omega_d (a/4)^{d-\beta}} \int_{\R^d}  |\nabla |f_j|(y)|\dy.
$$
Clearly, $M_\beta f (x) = M_\beta^I f(x)$ if $r_x \in \CR_x^\beta$ is such that $r_x \leq |x|/4$, whilst if $r_x \geq |x|/4 \geq a/4$ the inequality simply follows from Lemma \ref{basic}. In \cite[Remark 4.9]{BM2019}  it was shown that $\| \nabla M_\beta^I f_j - \nabla M_\beta^I f \|_{L^q(A(a,b))} \to 0$, and by Lemma \ref{lemma:ae convergence derivatives}, one has $ \nabla M_\beta^I f_j(x) \to \nabla M_\beta^I f(x)$. Moreover, note that Lemma \ref{lemma:convergence modulus in W11} implies
$$
\frac{1}{\omega_d (a/4)^{d-\beta}} \int_{\R^d}  |\nabla |f_j|(y)|\dy \to \frac{1}{\omega_d (a/4)^{d-\beta}} \int_{\R^d}  |\nabla |f|(y)|\dy
$$
as $j \to 0$. Consequently, the norm convergence also follows in $A(a,b)$ because this is a compact set. Therefore, the result follows from the a.e. convergence $\nabla M_\beta f_j \to \nabla M_\beta f$ ensured by Lemma \ref{lemma:ae convergence derivatives} and the Generalised dominated convergence theorem (Theorem \ref{thm:gdct}).
\end{proof}

We next show that there exists a small neighbourhood around the origin for which there is convergence. To this end, we use an idea recently introduced by González--Riquelme in \cite{GR}.

\begin{lemma}\label{bolita}
Let $0<\beta < 1$, $f \in W^{1,1}(\R^d)$ and $\{f_j\}_{j \in \N} \subset W^{1,1}(\R^d)$ radial functions such that $\| f_j - f \|_{W^{1,1}(\R^d)} \to 0$. Then, for every $\epsilon>0$ there exists an $\delta>0$ and $j_\epsilon > 0$ such that
\begin{equation*}
\| \nabla M^{}_\beta f_j - \nabla M^{}_\beta f \|_{L^q(B(0,\delta))}<\epsilon \ \ \quad  \text{for all} \, \ j\geq j_\epsilon,  
\end{equation*}
where $q=d/(d-\beta)$.
\end{lemma}

\begin{proof}
Fix $\lambda>0$ to be chosen later. Then, for every $\epsilon>0$ there exist $\delta:=\delta_{\epsilon, \lambda}>0$ and $j_{\epsilon}>0$ such that
$$
\int_{B(0,(\lambda+1)\delta)}|\nabla f|<\epsilon
 \qquad \text{and} \qquad \int_{B(0,(\lambda+1)\delta)}|\nabla f_j|<\epsilon
$$
for every $j\geq j_{\epsilon}$. For each $j \in \N$, define the sets
$$
E^{j}_1:=\{x\in B(0,\delta) \, : \, \exists \,   r_{x,j} \in \CR_{x, j}^\beta \, \,\, \text{with}\, \, \,  r_{x,j} \geq \lambda\delta\},
$$
and
$$
E^{j}_2:=\{x\in B(0,\delta) \, : \,  \exists \,  r_{x,j} \in \CR_{x, j}^\beta \, \, \, \text{with} \, \,\,   r_{x,j}< \lambda\delta\};
$$
the sets $E_1$ and $E_2$ are defined similarly with respect to the function $f$. 

On the one hand, in $E^{j}_1$, the lower bound on the radius yields
\begin{align*}
\int_{B(0,\delta)\cap E^{j}_1}|\nabla M_{\beta}f_j(x)|^{q}\dx \, &\leq \, \|\nabla f_j\|^{q-1}_1\int_{B(0,\delta)\cap E^{j}_1}\frac{1}{ \omega_d (\lambda\delta)^{d}}\int_{B_{x,j}}|\nabla f_j(y)|\dy\dx\\
&\leq \, \frac{1}{\lambda^{d}}\|\nabla f_j\|^{q}_1 \\
&\lesssim \, \frac{1}{\lambda^{d}}\|\nabla f\|^{q}_1 
\end{align*}
for all $j \geq j_{\epsilon}$, and the same holds in $B(0,\delta) \cap E_1$.

On the other hand, in $E^{j}_2$, the bound \eqref{eq:sobolev bound} yields
\begin{align*}
\int_{B(0,\delta)\cap E^{j}_2}|\nabla M_{\beta}f_j(x)|^{q}\dx \, &\leq \, \int_{B(0,\delta)}|\nabla M_{\beta}(f_j\chi_{B(0,(\lambda+1)\delta)})(x)|^{q} \dx \\
&\lesssim \, \|\nabla f_{j}\|^{q}_{L^{1}(B(0,(1+\lambda)\delta))}\\
&\leq \, \epsilon^q
\end{align*}
for all $j \geq j_{\epsilon}$, and the same holds in $B(0,\delta) \cap E_2$.

It then suffices to choose $\lambda>0$ large enough so that $\| \nabla f \|_1/\lambda^{d/q} < \epsilon$.
\end{proof}

It is clear that Lemmas \ref{anillo} and \ref{bolita} can be combined to obtain Proposition \ref{prop:convergence compact} and concluding then \eqref{eq:goal}, which is the desired continuity result.

\bibliography{Reference}
\bibliographystyle{amsplain}

\end{document}